\documentclass[a4paper,12pt,reqno]{amsart}
\usepackage{amssymb,amsfonts,amscd,latexsym}
\usepackage{enumerate}
\usepackage{amsmath, mathrsfs,bm}
\usepackage{color}
\usepackage{xy}
\xyoption{matrix} \xyoption{arrow}

\numberwithin{equation}{section}

\let\oldmarginpar\marginpar
\renewcommand\marginpar[1]{\-\oldmarginpar[\raggedleft\bf\footnotesize #1]%
{\raggedright\bf\footnotesize #1}}

\begingroup
\newtheorem{thm}{Theorem}[section]
\newtheorem{cor}[thm]{Corollary}
\newtheorem{prop}[thm]{Proposition}
\newtheorem{lem}[thm]{Lemma}

\theoremstyle{definition}

\newtheorem*{acknowledgement}{Acknowledgement}
\endgroup

\newcommand{\rada}{{\textup{rad}\kern.7pt A}}
\newcommand{\El}[1]{{\mathcal{E}\kern-2.6pt\ell({#1})}}
\newcommand{\Ell}{{\mathcal{E}\kern-2.6pt\ell}}
\newcommand{\ElA}{{\mathcal{E}\kern-2.6pt\ell(A)}}
\newcommand{\ElnA}{{\mathcal{E}\kern-2.6pt\ell_n(A)}}
\newcommand{\ElthA}{{\mathcal{E}\kern-2.6pt\ell_3(A)}}
\newcommand{\LE}{{\mathscr L}(E)}
\newcommand{\invA}{{\textup{Inv}\kern.5pt A}}
\newcommand{\LS}{{\bm L(S)}}
\newcommand{\RS}{{\bm R(S)}}
\newcommand{\VS}{{\bm V\mkern-2.5mu(S)}}
\newcommand{\VpS}{{\bm{V'}\mkern-2.5mu(S)}}
\newcommand{\LSrho}{{\bm L(S_\varrho)}}
\newcommand{\RSrho}{{\bm R(S_\varrho)}}
\newcommand{\VSrho}{{\bm V\mkern-2.5mu(S_\varrho)}}
\newcommand{\VpSrho}{{\bm{V'}\mkern-2.5mu(S_\varrho)}}

\newcommand {\N}{\mathbb{N}} 
\newcommand {\CC}{\mathbb{C}} 
\newcommand{\Dim}{{\sf dim}\,}
\newcommand{\lDim}{{\sf ldim}\,}

\newcommand{\spa}{{\rm span}}

\newcommand{\Tr}{{\sf Tr}}
\newcommand{\spn}[1]{{\|{#1}\|_\sigma}}

\def\C*{{\sl C*}-algebra}
\def\Cs*{{\sl C*}-subalgebra}

\newcommand {\be}{\begin{equation*}}
\newcommand {\ee}{\end{equation*}}
\newcommand {\beq}{\begin{eqnarray*}}
\newcommand {\eeq}{\end{eqnarray*}}

\begin{document}
\title[More elementary operators that are spectrally bounded]{More elementary operators that are\\ spectrally bounded}
\author{Nadia Boudi}
\address{D\' epartement de Math\' ematiques, Universit\' e Moulay Ismail, Facult\'{e} des Sciences, Mekn\`es, Maroc}
\email{nadia\_boudi@hotmail.com}
\author{Martin Mathieu}
\address{Pure Mathematics Research Centre, Queen's University Belfast, University Road, Bel\-fast BT7 1NN, Northern Ireland}
\email{m.m@qub.ac.uk}

\begin{abstract}
We discuss some necessary and some sufficient conditions for an elementary operator
$x\mapsto\sum_{i=1}^n a_ixb_i$ on a Banach algebra $A$ to be spectrally bounded.
In the case of length three, we obtain a complete characterisation when $A$
acts irreducibly on a Banach space of dimension greater than three.
\end{abstract}

\dedicatory{Dedicated to the memory of \hbox{T.\kern.7ptT.\kern1ptWest}.}
\date{\today}

\maketitle

\section{Introduction}\label{sect:intro}

\noindent
Let $A$ and $B$ be unital Banach algebras over the complex field~$\CC$. Let $r(x)$ denote the spectral radius of
an element~$x$ in $A$ or~$B$. We say a linear mapping \hbox{$T\colon A\to B$} is \textit{spectrally bounded\/} if,
for some constant $M\geq0$ and all $a\in A$, the estimate $r(Ta)\leq M\,r(a)$ holds.
This concept, together with its relatives \textit{spectrally isometric\/} (i.e., $r(Ta)=r(a)$ for all $a\in A$) and
\textit{spectrally infinitesimal\/} (i.e., $r(Ta)=0$ for all $a\in A$), was introduced in~\cite{Mat1} in order to initiate a systematic
investigation of mappings that had, on and off, been discussed in the literature; see, e.g.,~\cite{Aup} or~\cite{Pta}.
A number of fundamental properties of spectrally bounded operators can be found in~\cite{MaSc} and~\cite{MaSo}
while~\cite{MaSc2} contains a structure theorem for such operators defined on properly infinite von Neumann algebras.
Spectrally bounded operators also appear in connection with the noncommutative Singer--Wermer conjecture and with
Kaplansky's problem on invertibility-preserving operators; for details see~\cite{Mat2}.

An \textit{elementary operator\/} on~$A$ is a bounded linear operator $S\colon A\to A$ that can be written in the form
$Sx=\sum_{i=1}^n a_ixb_i$, $x\in A$ for some $a_1,\ldots,a_n$,  $b_1,\ldots,b_n\in A$.
These operators appear quite naturally in many contexts; for instance, if $A$ is finite dimensional and semisimple, every linear mapping
is of this form. In general, additional assumptions on the algebra and on an operator~$S$ may ``force'' the operator to be elementary:
a typical example is the innerness of a derivation $d\colon A\to A$, that is, $dx=ax-xa$ for some $a\in A$.
Properties of elementary operators have been studied under a vast variety of aspects; we refer the reader to~\cite{Ma92}
and~\cite{CuMa2} for an overview.

Despite the rich literature on spectrally bounded operators in general, and spectral isometries in particular, see, e.g.,
\cite{BrSe2, Co09, CoRe, Mat2, MaSo2, Sem} and the references contained therein, the supply of examples is still somewhat limited.
It is thus close at hand to ask which elementary operators are spectrally bounded, as these operators are given in a more concrete form.
Continuing our work started in~\cite{NaMa11}, we aim to provide further answers to this question in the present paper.
We shall discuss a number of necessary conditions which, in the case of length three, turn out to be sufficient too.
Our new approach exploits the relation with locally quasi-nilpotent elementary operators which, in the algebraic setting,
were studied in~\cite{NaMa13}; in fact, that paper should be read in conjunction with the present one.

In order to illustrate the ideas, let us assume  that the elementary operator $S\colon A\to A$ is spectrally infinitesimal.
Let $\varrho$ be an irreducible representation of $A$ on a Banach space~$E$. Since $S$ induces an elementary operator
$S_\varrho\colon\varrho(A)\to\varrho(A)$ via $S_\varrho\circ\varrho=\varrho\circ S$, $S_\varrho$ is spectrally infinitesimal too.
By Jacobson's density theorem \cite[Theorem~4.2.5]{Aup}, $\varrho(A)$ is a dense (i.e., $n$-transitive for all~$n$) algebra on~$E$,
and we can apply the setting of~\cite{NaMa13}. Suppose $\zeta\in E$ and $x\in A$ are such that
$\varrho(x)\varrho(b_ia_j)\zeta\subseteq\CC\zeta$ for all~$i,j$. It is easy to see that this implies that
\[
S_\varrho\varrho(x)\bigl(\spa\{\varrho(a_1)\zeta,\ldots,\varrho(a_n)\zeta\}\bigr)\subseteq\spa\{\varrho(a_1)\zeta,\ldots,\varrho(a_n)\zeta\};
\]
consequently, the restriction of
$S_\varrho\varrho(x)$ to the finite-dimensional invariant subspace $\spa\{\varrho(a_1)\zeta,\ldots,\varrho(a_n)\zeta\}$ has to be nilpotent
which then allows us to apply the theory developed in~\cite{NaMa13}. A first application of this method is presented in
Proposition~\ref{prop:infini_elops} and elaborations on this idea provide the main techniques for Section~\ref{sect:spbdd};
see, in particular, Lemma~\ref{spb-nilp}.

The general questions that we pursue in this context are as follows. Let the elementary operator $Sx=\sum_{i=1}^n a_ixb_i$, $x\in A$
on $A$ be given.
\begin{enumerate}[(a)]
\item Suppose $S$ is spectrally bounded.
\begin{enumerate}[(i)]
\item What properties of the coefficients $a_i$, $b_i$ can we derive?
\item Can we find an ``improved'' representation of $S$ in the sense that the new coefficients have better properties?
\end{enumerate}
\item Which conditions on the coefficients $a_i$, $b_i$ ensure that $S$ is spectrally bounded?
\end{enumerate}
After collecting a number of basic properties and tools in Section~\ref{sect:pres}, we give several answers to question~(a) above
in Section~\ref{sect:spbdd} culminating in Theorem~\ref{spb-Ger} which describes the size of various spaces associated
to the coefficients of the induced elementary operator in an irreducible representation of~$A$ in terms of the local dimension.
Specialising to the case of length two elementary operators we derive further properties at the end of this section
and also correct a small oversight in \cite[Theorem 3.5]{NaMa11} concerning an exceptional case that can appear in dimension two.

A full answer to both questions~(a) and~(b) is, at present, only available for elementary operators of short length.
It was given in~\cite{NaMa11} for length two and is provided for length three
under the assumption that $A$ acts irreducibly on a Banach space of dimension greater than three
in Section~\ref{sect:length3} below (spaces with smaller dimension need to be treated separately).
The formulation seems too technical to allow an extension to the general case so far.

\vspace{-.3cm}
\section{Prerequisites}\label{sect:pres}

\noindent
Throughout this paper, $A$ will denote a unital complex Banach algebra, and its group of invertible elements is written
as~$\invA$. We let $\rada$ stand for the Jacobson radical of~$A$, see~\cite[p.~34]{Aup}.
The algebra of all bounded linear operators on a Banach space~$E$ will be designated by~$\LE$.

An \textit{elementary operator\/} on $A$ is a bounded linear mapping $S\colon A\to A$ that can be written in the form
\begin{equation}\label{eq:elop-def}
Sx =\sum_{i=1}^n a_ixb_i\quad(x\in A),
\end{equation}
for some $a_i,\,b_i\in A$ and some $n\in\N$. Special cases are $L_a\colon x\mapsto ax$, $R_b\colon x\mapsto xb$ and
$M_{a,b}=L_aR_b$. Clearly the representation  of $S$ in a sum as in~\eqref{eq:elop-def} is not unique.
The smallest $n\in\N$ such that the non-zero elementary operator $S$ can be written as $S=\sum_{i=1}^n M_{a_i,b_i}$
is called \textit{the length of}~$S$ and will be abbreviated as~$\ell(S)$. We put $\ell(0)=0$. If $S=\sum_{i=1}^n M_{a_i,b_i}$
and $\ell(S)=n$ then, evidently, the sets $\{a_1,\ldots,a_n\}$ and $\{b_1,\ldots,b_n\}$ are linearly independent.

We will denote by $\ElA$ and $\ElnA$, respectively, the algebra of all elementary operators on~$A$ and
the set of all elementary operators of length~$n$, respectively.

Whenever convenient, we shall abbreviate an $n$-tuple $(a_1,\ldots,a_n)$ of elements of $A$ by~$\bm a$ and indicate that $S\in\ElA$
is written as ${S} =\sum_{i=1}^n M_{a_i, b_i}$ by $S=S_{\bm a,\bm b}$.
We shall further use the following notation for  $S=S_{\bm a,\bm b}$:
\begin{equation*}\label{eq:notation}
\begin{split}
\LS  &= \spa \{a_1, \ldots, a_n \},\\
\RS  &= \spa \{b_1, \ldots, b_n\},\\
\VS  &= \spa \{b_i a_j: 1 \leq i,j \leq n\},\text{and}\\
\VpS&=\VS+\CC I,
\end{split}
\end{equation*}
together with the abbreviations $S^*$ for $S_{\bm b,\bm a}$ and $\bm{ba}$ for $\sum_{i=1}^n b_ia_i$.

The way representations of $S\in\ElA$ as in~\eqref{eq:elop-def} are related to each other can be rather intricate;
however, we shall be content with representations arising from each other by linear combinations of the coefficients.
In this case, we have the following result the argument for which is standard but we include a proof in order to illustrate how to work with
different representations of the same elementary operator.
\begin{lem}\label{lem:various-repns}
Let $A$ be a unital Banach algebra,  and let ${S} =\sum_{i=1}^n M_{a_i, b_i}$ be an elementary operator on~$A$.
Suppose that ${S}= \sum_{i=1}^m M_{c_i, d_i}$, where $c_i \in \LS$ and $ d_i \in \RS$ for all $1 \leq i \leq m$.
Then $\sum_{i=1}^n b_i a_i - \sum_{i=1}^m d_i c_i \in \rada$.
\end{lem}
\begin{proof}
As ${S}$ leaves each primitive ideal $P$ of $A$ invariant and we have to show that, for every~$P$, we have
\[
\sum_{i=1}^n b_i a_i - \sum_{i=1}^m d_i c_i \in P,
\]
we can  assume that $A$ is primitive.

Without loss of generality, we may suppose that ${S}$ has length $n$ and that $\{c_1,\ldots,  c_n\}$ is linearly independent.
First assume that $n=m$. Write $c_i= \sum_{k=1}^n \alpha_{ik} a_k$, $1\leq i\leq n$. Then it follows from \cite[Theorem 5.1.7]{ArMa}, e.g.,
that $b_k= \sum_{i=1}^n \alpha_{ik} d_i$, $1\leq k\leq n$. Now it is easy to see that $\sum_{i=1}^n b_i a_i= \sum_{i=1}^n d_i c_i$.

Next suppose that $n <m$ and write
\begin{equation*}
c_j= \sum_{k=1}^n \alpha_{jk} c_k, \;\; \text {for } n+1 \leq j  \leq m.
\end{equation*}
Then
\begin{equation*}
{S}= \sum_{i=1}^n M_{c_i, d_i}+ \sum_{j=n+1}^m \sum_{k=1}^n  \alpha_{jk}M_{c_k, d_j}
\end{equation*}
which entails that
\begin{equation*}
{S}   = \sum_{i=1}^n M_{c_i,d_i}+ \sum_{k=1}^n \sum_{j=n+1}^m  M_{c_k, \alpha_{jk}d_j}
           = \sum_{k=1}^n M_{c_k,d_k+\sum\limits_{j=n+1}^m\!\alpha_{jk} d_j}.
\end{equation*}
Setting $d'_k= d_k+\sum_{j=n+1}^m \alpha_{jk}d_j$ we find, as shown above, that
$\sum_{k=1}^n d'_k c_k= \sum_{k=1}^n b_ka_k$.
Substituting the above expression for $d_k'$ back into this identity yields the desired conclusion.
\end{proof}
In the above argument we used implicitly that every primitive Banach algebra is centrally closed,
that is, its so-called extended centroid is equal to~$\CC$; the important consequence for us is that
we only have to work with linear combinations over the complex numbers.

We would like to use the fact that an elementary operator leaves every ideal of $A$ invariant to reduce
the task of describing spectrally bounded elementary operators to the case of primitive Banach algebras.
However, the induced elementary operator on the quotient of $A$ by a primitive ideal may not be spectrally bounded
without further assumptions. Resulting from this we will have to work explicitly with irreducible representations
in the following in order to apply the results obtained in~\cite{NaMa13} for locally quasi-nilpotent elementary operators on irreducible
algebras of operators.

Nevertheless, there is a class of Banach algebras (which includes all \C*s, for example) which allows us to reduce
the problem fully to primitive quotients. Recall that a Banach algebra $A$ is said to be an \textit{SR-algebra\/}
if the spectral radius formula holds in every quotient; that is, whenever $I$ is a closed ideal of $A$, for each $x\in A$ we have
\begin{equation}\label{eq:spec-rad-form}
r(x+I)=\inf_{y\in I} r(x+y).
\end{equation}
Whenever $T\colon A\to B$ is a linear mapping, $I$ is a closed ideal of $A$ and $J$ is a closed ideal of $B$ containing $TI$,
we can define an induced linear mapping $\hat T\colon\hat A=A/I\to\hat B=B/J$ by $\hat T(x+I)=Tx+J$, $x\in A$.
If $T$ is spectrally bounded, we define the \textit{spectral norm\/} $\spn T$ of $T$ as the smallest $M\geq0$ such that
$r(Tx)\leq M\,r(x)$ for all $x\in A$, see~\cite{MaSc}. (Note that this is in general not a norm!)

The standard argument for induced bounded linear operators on quotient spaces enables us to prove the following result.
\begin{prop}\label{prop:sr-quotients}
Let $T\colon A\to B$ be a spectrally bounded operator from the unital SR-algebra $A$ into the unital Banach algebra~$B$.
Suppose that $I$ is a closed ideal of $A$ and $J$ is a closed ideal of $B$ containing~$TI$. Then $\hat T\colon\hat A\to\hat B$
is spectrally bounded with $\spn{\hat T}\leq\spn T$.
\end{prop}
\begin{proof}
Let $x\in A$ and, for given $\varepsilon>0$, let $y\in I$ be such that $r(x+y)\leq r(x+I)+\varepsilon$. Then
\begin{equation*}
\begin{split}
r\bigl(\hat T(x+I)\bigr) &=r(Tx+J)\leq\inf_{z\in J}r(Tx+z)\\
                                              &\leq r(Tx+Ty)\leq\spn T\,r(x+y)\\
                                              &\leq\spn T\bigl(r(x+I)+\varepsilon\bigr),
\end{split}
\end{equation*}
hence $r\bigl(\hat T(x+I)\bigr)\leq\spn T\,r(x+I)$ which yields the claim.
\end{proof}
As an illustration of how smooth the arguments become in the case of \C*s, and to contrast the more elaborate
techniques we have to use in the general situation, we formulate the following consequence which is a special
case of Proposition~\ref{tr-spb} below.
\begin{cor}\label{cor:cstar-quotients}
Let $S\in\ElA$ for a unital \C*~$A$ be spectrally bounded. If $S=S_{\bm a,\bm b}$ then $\bm{ba}\in Z(A)$,
the centre of~$A$.
\end{cor}
\begin{proof}
Let $\varrho$ be an irreducible representation of $A$ on a Hilbert space $H$ and put $P=\ker\varrho$.
As $SP\subseteq P$ and since $A$ is an SR-algebra, by~\cite{Ped76}, see also~\cite{MW79},
the induced elementary operator $S_\varrho$ defined by $S_\varrho\circ\varrho=\varrho\circ S$ is
spectrally bounded with $\spn{S_\varrho}\leq\spn S$ by Proposition~\ref{prop:sr-quotients} above.
As a result, we can assume that $A$ acts irreducibly on~$H$ and aim to show that $\sum_{i=1}^n b_ia_i\in\CC1$.
Slightly simplified arguments like those used in the proof of Proposition~\ref{tr-spb} (which we do not spell out here
to save space) allow us to arrive at the desired conclusion.
\end{proof}

For a unital Banach algebra $A$, let $\mathcal Z(A)$ denote the centre modulo the radical, that is, the inverse image
of the centre of $A/\rada$ under the canonical epimorphism. The next proposition determines when the basic block
of an elementary operator, the two-sided multiplication is a homomorphism.
\begin{prop}\label{mult-hom}
Let $A$ be a unital Banach algebra, let $u,v\in A$ and put $e=vu$. The two-sided multiplication $M_{u,v}$ is a homomorphism
modulo $\rada$ if and only if $e$ is an idempotent in $\mathcal Z (A)$ and $M_{u,v}(1-e)A\subseteq\rada$.
\end{prop}
\begin{proof}
To establish the ``only if''-part, let $\varrho$ be an irreducible representation of $A$ on a Banach space~$X$.
Suppose that there exists $\zeta \in X$ such that $\{ \varrho(vu) \zeta, \zeta\}$ is linearly independent. Choose $x,y \in A$ such that
\begin{equation*}
\varrho (xvu) \zeta=0,\;  \varrho (x) \zeta= \zeta\  \text {and } \varrho(yvu) \zeta= \zeta.
\end{equation*}
Then
\begin{equation*}
\varrho (M_{u,v} (xy) u) \zeta= \varrho (u) \zeta\text{ and }\varrho((M_{u,v} (x) M_{u,v} (y))u) \zeta=0,
\end{equation*}
a contradiction. Consequently, $\{ \varrho(vu) \zeta, \zeta\}$ is linearly dependent for every $\zeta \in X$. As a result, $\varrho (vu) \in \CC I$,
say $\varrho (vu)=\lambda I$. In order to show that  $\lambda\in \{0,1\}$, let $x,y\in A$ be arbitrary. Then, by hypothesis,
\begin{equation}\label{eq:homo}
\varrho(uxyv)=\varrho(uxv\,uyv)\quad{}\Longrightarrow{}\quad\varrho((1-\lambda)\,uxyv)=0.
\end{equation}
Specialising to $x=y=1$ and multiplying on the left with~$v$ and on the right with~$u$ we find that $(1-\lambda)\lambda^2=0$
which yields the claim. We conclude that $e$ is an idempotent in~$\mathcal Z(A)$.

Moreover, by~\eqref{eq:homo}, $\varrho(u(1-e)xv)=\varrho((1-\lambda)\,uxv)=0$ for all $x\in A$ which entails that
$M_{u,v}(1-e)A\subseteq\rada$.

Conversely, the hypotheses on $e$ and $M_{u,v}$ imply that
\[
uxyv+\rada=uxeyv+\rada +ux(1-e)yv+\rada=uxv\,uyv+\rada\quad(x,y\in A)
\]
which is the desired assertion.
\end{proof}

It follows from the above proposition that, if $u,v$ are elements of a unital semisimple Banach algebra~$A$ and $e=vu$ is a central
idempotent in $A$, then $M_{u,v}$ can be decomposed as
\[
M_{u,v}=M_{u_1,v_1}+M_{u_2,v_2}
\]
where $M_{u_1,v_1}$ is a homomorphism on~$A$ and $M_{u_2,v_2}(eA)=0$. This is easily verified by putting
$u_1=eu$, $v_1=ev$, $u_2=(1-e)u$ and $v_2=(1-e)v$.

The next result shows how to build a new spectrally bounded operator out of a number of given ones;
it generalises \cite[Lemma~2.1]{NaMa11}.
\begin{prop}\label{prop:spb-hom}
Let\/ $A$ be a  unital Banach algebra, and let\/ $T_1,\ldots, T_n$  be linear mappings on $A$ such that, for each~$i$,
$T_i$ is a homomorphism or $(T_i x)^2 \in\rada$ for every $x\in A$.
Suppose that $(T_i y) (T_jx)\in\rada$ for every $x, y\in A$ and for all $i>j$.
Then, for all\/ $\lambda_1, \ldots,  \lambda_n \in\CC$, the mapping\/ $T=\sum_{i=1}^n \lambda_i T_i$ is spectrally bounded.
\end{prop}
\begin{proof}
Suppose that for $1\leq t\leq r$, $T_{i_t}$ is a homomorphism and for $r+1\leq t\leq n$, $(T_{i_t} x)^2\in\rada$ for all $x\in A$.
Let $\lambda_1, \ldots,\lambda_n \in\CC$ and suppose, without loss of generality, that $\lambda_{i_t}=1$ for $r+1\leq t\leq n$.
Put $\lambda_t:=\lambda_{i_t}$ for $1\leq t\leq r$ and fix $x\in A$.
Choose a non-zero complex number $\alpha\not\in\bigcup_{t=1}^r \sigma(\lambda_t x)$.
Fix $1\leq t\leq r$ and put  $y_t=\frac{\lambda_t}{\alpha}x\bigl(\frac{\lambda_t}{\alpha}x-1\bigr)^{-1}$. Then
\begin{equation*}
\frac{\lambda_t}{\alpha}\,x +y_t-y_t \frac{\lambda_t}{\alpha}\,x=0
\end{equation*}
and thus
\begin{equation*}
\frac{\lambda_t}{\alpha}\,T_{i_t}x +T_{i_t}y_t - (T_{i_t}y_t)\frac{\lambda_t}{\alpha}\,T_{i_t}x=0.
\end{equation*}
We compute that
\begin{equation}\label{eq:first}
\bigl(1- \sum_{t=1}^r T_{i_t}y_t\bigr) \bigl(1- \sum_{t=1}^r\frac{\lambda_t}{\alpha}\,T_{i_t}x-\sum_{t=r+1}^n\frac{1}{\alpha}T_{i_t}x\bigr)
= 1+q,
\end{equation}
where
\be
q=\sum_{t,s=1,\, t\neq s}^r \frac{\lambda_s}{\alpha} \, (T_{i_t}y_t)(T_{i_s}x)
     - \sum_{t=r+1}^n \frac{1}{\alpha}T_{i_t}x + \sum_{t=1}^r\sum_{s=r+1}^n \frac{1}{\alpha}(T_{i_t}y_t)(T_{i_s}x).
\ee
As in the proof of Proposition~2.3 in~\cite{NaMa13}, we show by induction on $n$ that $q\in\rada$.
This implies that $1+q$ is invertible which, together with
identity~\eqref{eq:first}, entails that $\alpha-\bigl(\sum_{t=1}^r\lambda_t T_{i_t}+\sum_{t=r+1}^nT_{i_t}\bigr) (x)$ is left invertible.
Since the boundary of the spectrum of $Tx$ is contained in the left approximate point spectrum, it follows
that $ \alpha\notin\partial\sigma(Tx)$  and thus
\[
\partial\sigma (Tx) \subseteq \bigcup_{t=1}^r \sigma(\lambda_tx) .
\]
This implies that $r(Tx)\leq\max\limits_{1\leq i\leq n}\{|\lambda_i|\}\,r(x)$ for each $x\in A$. The  proof is complete.
\end{proof}

As a consequence we obtain a sufficient criterion for spectral boundedness of an elementary operator.
\begin{cor}\label{cor:sufficient-cond}
Let\/ $A$ be a  unital Banach algebra, and let\/ $S\in\ElnA$. If\/ $S$ can be written as\/ $S=\sum_{i=1}^nM_{u_i,v_i}$
with\/ $v_iu_j\in\rada$ for all $i>j$ and\/ $v_iu_i\in{\mathcal Z(A)}$ for all~$i$ then\/ $S$ is spectrally bounded.
\end{cor}
\begin{proof}
Let $\varrho$ be an irreducible representation of $A$; then $S_\varrho\in\El{\varrho(A)}$ is given by
$S_\varrho=\sum_{i=1}^nM_{\varrho(u_i),\varrho(v_i)}$. Since, for each~$i$, $\varrho(v_i)\varrho(u_i)=\mu_i1\in\CC1$,
the two-sided multiplication $M_{\varrho(u_i),\varrho(v_i)}$ is spectrally bounded with
$\|M_{\varrho(u_i),\varrho(v_i)}\|_\sigma\leq|\mu_i|$ (see the first paragraph of Section~\ref{sect:spbdd}).
If $\mu_i=0$ then $\bigl(M_{\varrho(u_i),\varrho(v_i)}\varrho(x)\bigr)^2=0$
for all $x\in A$. Otherwise, $\mu_i^{-1}M_{\varrho(u_i),\varrho(v_i)}$ is a homomorphism.
Setting $T_i=M_{\varrho(u_i),\varrho(v_i)}$, $1\leq i\leq n$ the hypothesis $\varrho(v_i)\varrho(u_j)=0$, $i>j$
yields $T_i\varrho(y)\,T_j\varrho(x)=0$ for all $x,y\in A$ and $i>j$. Thus we can apply Proposition~\ref{prop:spb-hom} to obtain
that $S_\varrho=\sum_{i=1}^n\lambda_iT_i$ is spectrally bounded with
$\|S_\varrho\|_\sigma\leq\smash{\max\limits_{1\leq i\leq n}}|\lambda_i|\vphantom{I_{I_I}}$,
where $\lambda_i=1$ if $\mu_i=0$ and $\lambda_i=\mu_i$ otherwise.

Put $\gamma=1+\smash{\max\limits_{1\leq i\leq n}}\|v_iu_i\|$. Then, for all $x\in A$,
$r(Sx)=\sup_\varrho r(S_\varrho\varrho(x))\leq\gamma\, r(x)$, compare \cite[p.~4]{NaMa11}.
As a result, $S$ is spectrally bounded.
\end{proof}

Let $V$  be a finite-dimensional subspace of~$L(X)$, the space of all linear mappings on a vector space~$X$.
Recall that $\lDim V= \max \{ \Dim V \zeta: \zeta \in X \}$ is the \textit{local dimension of}~$V$, and
that $V$ (or, equivalently, a basis of~$V$) is said to be \textit{locally linearly dependent\/}  if $\lDim V < \Dim V$.
In the case that $\lDim V= \Dim V$, any vector satisfying $\Dim V \zeta=\Dim V$ is called a \textit{separating vector of}~$V$.

The following lemma follows from \cite[Lemma 2.1]{BrSe}; see also~\cite{GoLaWo}. It will be used at
various instances in Section~\ref{sect:spbdd}.

\begin{lem}\label{free}
Let $X$ be a vector space and let $V_1, \ldots, V_k$ be finite-dimensional subspaces of $L (X)$.
Then there exists $\zeta \in X$ such that $\Dim V_i \zeta= \lDim V_i$ for all $1 \leq i \leq k$.
\end{lem}

\section{Spectrally bounded elementary operators}\label{sect:spbdd}

\noindent
Probably the first result on spectrally bounded elementary operators is Pt\'ak's theorem \cite[Proposition~2.1]{Pta}
showing that $L_a$ (equivalently, $R_a$) is spectrally bounded if and only if $a\in\mathcal Z(A)$.
It follows immediately from this that $M_{a,b}$ is spectrally bounded if and only if $ba\in\mathcal Z(A)$.
It was shown in \cite[Theorem~B]{CuMa} that $L_a-R_b$ is spectrally bounded if and only if both $a$ and $b$
belong to~$\mathcal Z(A)$; however, the condition $a-b\in\mathcal Z(A)$ is no longer sufficient.
If a length two elementary operator $M_{a,b}+M_{c,d}$ is spectrally bounded then $ba+dc\in\mathcal Z(A)$,
by \cite[Theorem~3.5]{NaMa11}; the latter theorem also gives a sufficient condition for spectral boundedness
of $M_{a,b}+M_{c,d}$. We will extend the necessary condition to arbitrary elementary operators in Proposition~\ref{tr-spb}
below; it seems difficult to give concise necessary \textit{and\/} sufficient conditions in general though.

We begin our discussion in this section by looking at spectrally infinitesimal elementary operators.

\begin{prop}\label{prop:infini_elops}
Let $A$ be a unital Banach algebra, and let $S=S_{\bm a,\bm b}$ be an elementary operator on~$A$.
Suppose that $r( Sx)=0$ for every  $x \in A$. Then $\bm b\bm a\in\rada$.
\end{prop}
\begin{proof}
Let ${S}_\varrho\colon\varrho(A)\to\varrho(A)$ denote the induced elementary operator, where $\varrho$
is an irreducible representation of~$A$ on a Banach space~$E_\varrho$. Then $r\bigr(S_\varrho\varrho(x)\bigl)=0$ for all $x\in A$.
As indicated in the Introduction, whenever $\zeta\in E_\varrho$ and $x\in A$ are such that
$\varrho(x)\VSrho\zeta\subseteq\CC\zeta$ then $S_\varrho\LSrho\zeta\subseteq\LSrho\zeta$ and thus
${S_\varrho}_{|\LSrho\zeta}$ is nilpotent.
As a result, the assumptions in \cite[Proposition~3.1]{NaMa13} are satisfied and it thus follows
that $\varrho\bigl(\sum_{i=1}^n b_i a_i\bigr)=0$.
\end{proof}

In the spectrally bounded case, we have the following more general situation.
Let us point out once again that, when $S\in\ElA$ is spectrally bounded and $\varrho$ is an irreducible representation of~$A$,
the elementary operator $S_\varrho\colon \varrho(A)\to\varrho(A)$ may or may not be spectrally bounded;
nevertheless the operator $\varrho\circ S$ \textit{is\/} spectrally bounded (with $\spn{\varrho\circ S}\leq\spn S$).

\begin{prop}\label{tr-spb}
Let $A$ be a unital Banach algebra, and let  $S=S_{\bm a,\bm b}$ be a spectrally bounded elementary operator on~$A$.
Then $\bm b\bm a\in\mathcal Z (A)$.
\end{prop}
\begin{proof}
Let $\varrho$ be an irreducible  representation of $A$ on a  Banach space~$E$.
Suppose that there exist $\zeta\in E$ and $t\in\{1, \ldots, n\}$ such that the set $\{\varrho (b_t a_t) \zeta, \zeta\}$
is linearly independent (that is, $\varrho(b_t a_t)\notin\CC1$).
Set $s=\Dim\LSrho\zeta$.
As in the proof of \cite[Proposition~3.1]{NaMa13}, we assume without loss of generality that
$\{\varrho(a_1)\zeta, \ldots, \varrho(a_s)\zeta\}$ is linearly independent and $\varrho (a_t) \zeta=0$ for $t >s$.
Choose $i_1, \ldots, i_r$ such that $\Xi=\{\zeta,\varrho(b_{i_1}a_{i_1})\zeta,\ldots,\varrho(b_{i_r} a_{i_r})\zeta\}$ is linearly independent,
$r$ being maximal (in particular, $\varrho(b_{i_t}a_{i_t})\notin\CC1$ for all $1\leq t\leq r$).
Fix $j \in \{1, \ldots, r\}$. By Sinclair's theorem \cite[Corollary~4.2.6]{Aup}, for each $k\in\N$ there is $x_k\in\invA$ such that
\begin{equation*}
\varrho ( x_k ) \zeta= \frac{1}{k} \varrho( b_{i_j} a_{i_j}) \zeta,\;
\varrho( x_k b_{i_j} a_{i_j}) \zeta = \zeta, \text{ and }
\varrho (x_k b_{i_t} a_{i_t}) \zeta = \varrho (b_{i_t} a_{i_t}) \zeta \ \,(1\leq t\leq r,\,t \neq j).
\end{equation*}
Choose $y \in A$ such that
\begin{equation*}
\varrho( y) \zeta=\varrho ( b_{i_j} a_{i_j}) \zeta \quad \text{and} \quad \varrho (y b_{i_t} a_{i_t}) \zeta =0 \quad(1 \leq t \leq r).
\end{equation*}
Then
\begin{equation}\label{eq:eigenvalues}
\varrho (x_k^{} y x_k^{-1} b_{i_j} a_{i_j}) \zeta= k\,\zeta
\quad\text{and}\quad
\varrho( x_k^{} y x_k^{-1} b_{i_t} a_{i_t}) \zeta=0\quad(1\leq t\leq r,\,t \neq j).
\end{equation}
Replacing the above-chosen $x_k$ and $y$ by ones which, in addition, satisfy
\begin{equation}\label{eq:better-xk-y}
\varrho(x_kb_ia_\ell)\zeta=\varrho(b_ia_\ell)\zeta\quad\text{and}\quad
\varrho(yb_ia_\ell)\zeta=0
\end{equation}
for all $1\leq i,\ell\leq n$, $i\ne\ell$ such that $\varrho(b_ia_\ell)\zeta\notin\spa\,\Xi$,
we can also assume that
\begin{equation}\label{eq:prop-V}
\varrho(x_k^{} y x_k^{-1})\VSrho\zeta\subseteq\CC\zeta.
\end{equation}

Now let $J= \{1, \ldots, n\}\setminus\{i_1, \ldots, i_r\}$ and write, for all $i \in J$,
\[
\varrho (b_i a_i) \zeta= \alpha_i \zeta + \sum_{t=1}^r \alpha_t^i \,\varrho (b_{i_t} a_{i_t}) \zeta.
\]
Then, by~\eqref{eq:eigenvalues}, $\varrho(x_k^{} y x_k^{-1} b_{i} a_{i}) \zeta= \alpha_j^i k\,\zeta$ for all $i \in J$.
Since $S_\varrho\varrho(x_k^{} y x_k^{-1} )$ leaves the finite-dimensional subspace $\LSrho\zeta$ invariant, by~\eqref{eq:prop-V},
we obtain
\begin{equation*}
\Tr\bigl(S_\varrho\varrho(x_k^{} y x_k^{-1} )_{|\LSrho\zeta}\bigr)= k \bigl(1+ \sum_{i \in J}\alpha_j^i \bigr),
\end{equation*}
where $\Tr$ denotes the trace on $\LSrho\zeta$.
Because the trace of a linear mapping $u$ on $\LSrho\zeta$ is dominated by $s\cdot r(u)$ and thus
\begin{equation*}
k \bigl(1+ \sum_{i \in J}\alpha_j^i \bigr)\leq s\cdot r\bigl(S_\varrho\varrho(x_k^{} y x_k^{-1})\bigr)
                                                                              = s\cdot r(\varrho S(x_k^{} y x_k^{-1}))
                                                                              \leq s\cdot\spn Sr(y)\quad(k\in\N),
\end{equation*}
we must have $1+ \sum_{i \in J}\alpha_j^i =0$. As
\begin{equation*}
\begin{split}
\sum_{i=1}^n \varrho(b_i a_i )\zeta= \sum_{i=1}^s \varrho(b_i a_i )\zeta
   &= \sum_{t=1}^r \varrho(b_{i_t} a_{i_t})\zeta + \sum_{i \in J} \varrho(b_i a_i)\zeta\\
   &= \sum_{t=1}^r \varrho(b_{i_t} a_{i_t})\zeta+ \sum_{i \in J} \bigl(\alpha_i\zeta+\sum_{t=1}^r \alpha_t^i \varrho(b_{i_t} a_{i_t})\zeta\bigr)\\
   &= \sum_{t=1}^r \bigl(1+ \sum_{i \in J} \alpha_t^i\bigr) \varrho(b_{i_t} a_{i_t})\zeta + \sum_{i \in J} \alpha_i\zeta,
\end{split}
\end{equation*}
we conclude that $\sum_{i=1}^n \varrho (b_i a_i) \zeta \in \mathbb C \zeta$. We have thereby shown  that $\textup{id}_E$ and
$\sum_{i=1}^n \varrho (b_i a_i)$ are locally linearly dependent. Thus $\sum_{i=1}^n \varrho (b_i a_i) \in \CC 1$, as desired.
\end{proof}

The following lemma is a key result enabling us to relate spectral boundedness to local quasi-nilpotency and thus to make use
of the results in~\cite{NaMa13}.

\begin{lem}\label{spb-nilp}
Let $A$ be a unital Banach algebra, and let $S$ be  a spectrally bounded  elementary operator on~$A$.
Let $\varrho$ be an irreducible representation  of $A$ on a Banach space~$E$.
Then, for every $\zeta \in E$ and $x\in A$ satisfying
\begin{equation*}
\varrho(x)\VSrho\zeta \subseteq\mathbb C \zeta \quad \text{and} \quad \varrho (x)\zeta=0,
\end{equation*}
the map $S_\varrho\varrho(x)_{|\LSrho\zeta}$ is nilpotent.
\end{lem}
\begin{proof}
Suppose that there exist $\zeta\in E$ and $x\in A$ satisfying $\varrho(x)\VSrho\zeta\subseteq \CC \zeta$, $\varrho(x)\zeta=0$ and
$S_\varrho\varrho(x)_{|\LSrho\zeta}$ is not nilpotent.  Set $s= \Dim\LSrho\zeta$.
Choose $u \in \LS$ such that $\varrho (u) \zeta \neq 0$ and $S_\varrho\varrho(x)\varrho(u)\zeta= \lambda\varrho(u)\zeta$,
where $\lambda$ is a non-zero eigenvalue of $S_\varrho\varrho(x)_{|\LSrho\zeta}$.
Write $S= M_{u,v}+ \sum_{i=1}^{n-1} M_{u_i,v_i}$.
By the same argument as in the proof of Proposition~\ref{tr-spb}, we can assume that $\varrho(u_i) \zeta=0$ for $i > s-1$.  Therefore,
\begin{equation*}
\varrho (uxvu) \zeta +\sum_{i=1}^{s-1} \varrho(u_i xv_iu )\zeta = \lambda \varrho( u) \zeta.
\end{equation*}
Since $\{\varrho (u) \zeta, \varrho (u_1) \zeta,\ldots, \varrho(u_{s-1}) \zeta   \}$ is linearly independent and
\[
\varrho(x)\,\{\varrho(vu)\zeta,\varrho(v_1u)\zeta,\ldots,\varrho(v_{s-1}u)\zeta\}\subseteq\CC\zeta,
\]
we get
\begin{equation}\label{00}
\varrho (xvu ) \zeta= \lambda \zeta  \quad \text{and} \quad \varrho (xv_i u) \zeta=0\quad(1 \leq i \leq s-1).
\end{equation}
Our assumption on $\zeta$ and $x$ implies that  $\{\zeta, \varrho (vu) \zeta \}$ is linearly independent.
Let $r\leq s-1$ be maximal such that
$\{\zeta, \varrho (vu) \zeta, \varrho( v_{i_1}u ) \zeta, \ldots, \varrho (v_{i_r}u )\zeta\}$ is linearly independent.
By the equations in \eqref{00}, we  infer that
$\varrho (v_iu )\zeta \in \spa \{\zeta, \rho (v_{i_1}u) \zeta, \ldots,\rho (v_{i_r}u )\zeta \}$
for each $1 \leq i \leq s-1$. Let $x_k \in\invA$, $k\in\N$ be such that
\begin{equation*}
\varrho(x_k)\zeta= \frac{1}{k}\varrho(vu)\zeta,\; \varrho(x_k vu )\zeta= \zeta \quad \text{and}\quad
           \varrho (x_k v_{i_t}u) \zeta=\varrho (v_{i_t} u ) \zeta \quad(1 \leq t \leq r).
\end{equation*}
Choose $y \in A$ such that
\begin{equation*}
\varrho (y) \zeta=\varrho (vu) \zeta \quad \text{and} \quad\varrho (y vu) \zeta= \varrho (y v_{i_t} u) \zeta= 0\quad(1 \leq t \leq r).
\end{equation*}
Then
\begin{equation*}
\rho (x_k^{} y x_k^{-1}vu) \zeta= k \zeta\quad\text{and}\quad\varrho(x_k^{} y x_k^{-1}v_{i_t}u) \zeta= 0\quad(1 \leq t \leq r)
\end{equation*}
and therefore $\varrho(x_k^{} y x_k^{-1}v_{i}u) \zeta =0$ for all $1 \leq i \leq s-1$.This entails that
\begin{equation*}
\varrho\,(S (x_k^{} y x_k^{-1})\, u) \zeta = k \varrho(u)\zeta\quad(k\in\N),
\end{equation*}
which contradicts our assumption $r(\varrho\,Sy)\leq\spn S\,r(y)$.
\end{proof}
In the situation of the above lemma, the following notation is useful.
For $\zeta\in E$, denote by $\pi\colon\CC\zeta\to\CC$, $\pi(\zeta)=1$ the canonical map.
For a linear mapping $y$ on $E$ and a basis $\{\zeta_1,\ldots,\zeta_k\}$ of a $y$-invariant subspace of~$E$,
let $M\bigl(y,\{\zeta_1,\ldots,\zeta_k\}\bigr)$ denote the corresponding matrix representation with respect to $\{\zeta_1,\ldots,\zeta_k\}$.
For instance, if $S=\sum_{i=1}^n  M_{a_i,b_i}\in\ElA$, $\varrho$ is an irreducible representation of $A$ on $E$ and
$\{\varrho(a_1)\zeta, \ldots, \varrho(a_n)\zeta\}$ is linearly independent, then
\[
M\bigl(S_\varrho\varrho(x), \{\varrho(a_1)\zeta, \ldots, \varrho(a_n)\zeta\}\bigr)= (\pi(\varrho(x b_i a_j)\zeta))_{1 \leq i, j \leq n}
\]
for every $\zeta\in E$ and $x\in A$ such that $\varrho(x)\VSrho\zeta\subseteq\CC\zeta$.

\smallskip
Let $S$ be an elementary operator of length $n$ on an algebra~$A$.
For each $u \in\LS$, we denote by $S\langle u \rangle$ the set
\be
S\langle u \rangle = \{v \in\RS\mid \text{there exists } S_{n-1}\in\Ell_{n-1} (A)\text{ such that }S= M_{u,v} + S_{n-1} \}.
\ee

\begin{prop}\label{spb-notqu}
Let $A$ be a unital Banach algebra, and let $S$ be  a spectrally bounded elementary operator of length $n$ on~$A$.
Let $\varrho$ be an irreducible representation of $A$ on a Banach space~$E$.
Suppose that $\Dim\VpSrho= \lDim\VpSrho$. Then either for every $\zeta \in E$ and $x \in A$ satisfying
$\varrho(x)\VSrho\zeta \subseteq \mathbb C \zeta$,  we have\/ $S_\varrho\varrho(x)_{|\LSrho\zeta}$ is nilpotent
or there exist ${S}_1 \in \Ell_{1} (A)$, ${S}_2 \in \Ell_{n-1} (A)$  such that $\varrho\,{S}_1$ is scalar multiple of a
homomorphism and ${S}= {S}_1+ {S}_2$.
\end{prop}
\begin{proof}
Suppose that there exist $\zeta\in E$ and $x\in A$ satisfying $\varrho(x)\VSrho\zeta \subseteq\mathbb C \zeta$ and
$S_\varrho\varrho(x)_{|\LSrho\zeta}$ is not nilpotent.  Choose $u\in\LS$ such that $\varrho(u)\zeta\neq 0$ and
$S_\varrho\varrho(xu)\zeta= \lambda\varrho(u )\zeta$, where $\lambda$ is a non-zero eigenvalue of
$S_\varrho\varrho(x)_{|\LSrho\zeta}$. We claim that  we can assume that $\zeta$ is a separating vector of $\VpSrho$.
Indeed, let $\zeta'$ be a separating vector of $\VpSrho$. If $\{\varrho (v_{i_t}u_{j_t}) \zeta\}_{1 \leq t \leq r}$ is
linearly independent, $r$ being maximal, then $\{\varrho (v_{i_t}u_{j_t}) \zeta' \}_{1 \leq t \leq r}$ is linearly independent too.
Hence we can choose $y\in A$ so that $\varrho (y v_iu_j) \zeta'= (\pi \varrho (x v_i u_j) \zeta) \zeta'$, and consequently,
$S_\varrho\varrho(yu) \zeta'= \lambda \varrho (u) \zeta'$. The claim is proved.
Next set $s= \Dim\LSrho\zeta$ and write $S=M_{u,v}+ \sum_{i=1}^{n-1} M_{u_i,v_i}$.
Then we may assume that $\varrho (u_i) \zeta=0$ for $s \leq i \leq n-1$. We have
\begin{equation*}
\varrho \bigl(uxvu  +\sum_{i=1}^{s-1} u_i xv_iu\bigr) \zeta = \lambda \varrho (u) \zeta.
\end{equation*}
Suppose, for instance,  that $\{\varrho (vu), \varrho (v_1 u), \ldots, \varrho (v_r u) \}$ is linearly independent,
$r$ being maximal and $r \leq s-1$. Our assumption on $x$ and $\zeta$ implies that for each
$1 \leq i \leq s-1$, $\varrho (v_i u) \in \spa \{\varrho (v_1 u), \ldots, \varrho (v_r u) \}$.
We write $I=\textup{id}_E$ and distinguish two cases:

\smallskip\noindent
Case 1: $I \in \spa \{ \varrho (vu), \varrho (v_1 u), \ldots, \varrho (v_ru) \}$.

\noindent
Suppose for a moment that $ I \in \spa \{ \varrho( v_1 u), \ldots, \varrho (v_r u)\}$.
Since $\varrho (x v_t u) \zeta=0$, we must have $\varrho(x) \zeta=0$. Using Lemma~\ref{spb-nilp}, we get  a contradiction.
Thus, there exist a non-zero complex number $\gamma$ and complex numbers $\tau_t$ such that
$I= \varrho (\gamma vu + \sum \tau_t v_t u)$. Set  $v'=v- \sum \frac{\tau_t}{\gamma} v_t$.
Then  $v' \in {S} \langle u \rangle$, $\varrho (v'u) \in \CC I$, and clearly $\varrho (v'u ) \neq 0$.

\smallskip\noindent
Case 2:  $I \not \in \spa \{ \varrho (vu), \varrho (v_1 u), \ldots, \varrho (v_r u)\}$.

\noindent
Since $ \lDim\VpSrho= \Dim\VpSrho$, $\{ \zeta, \varrho( vu) \zeta, \ldots, \varrho(v_t u) \zeta \}$
is linearly independent. As in the proof of Lemma~\ref{spb-nilp},  we get a contradiction using suitable $x_k\in\invA$ and $y\in A$.
We have thereby shown that this case cannot occur.
\end{proof}
\begin{lem}\label{lem:ldimV'=1}
Let $A$ be a unital Banach algebra, and let $S$ be  an  elementary operator on~$A$.
Let $\varrho$ be an irreducible representation of $A$ on a Banach space~$E$.
Suppose that $\ell(S_\varrho)=n$ and\/  $\lDim\VpSrho= 1$. Set $\lDim\LSrho=r$.
Then $\VSrho\subseteq \CC I$ and $S_\varrho$ admits a representation of the form $S_\varrho= \sum_{i=1}^n M_{u_i,v_i}$, where
\be
(v_i u_j)_{1 \leq i,j \leq n}= \left(
                                   \begin{array}{cc}
                                     T & 0 \\
                                     * & 0 \\
                                   \end{array}
                                 \right)
\ee
 and $T$ is a triangular  matrix of order~$r$. Moreover, $S_\varrho^* S_\varrho^{}\in\CC I$.
\end{lem}
\begin{proof}
Write $S_\varrho=\sum_{i=1}^n {M_{a_i,b_i}}$ for some $a_i,b_i\in\varrho(A)$.
It follows from our assumption on $\VpSrho$ that
$b_ia_j\in \CC I $ for each~$i,j$. Hence $\VSrho\subseteq\CC I$.  Pick a vector $\zeta \in E$ such that $\LSrho \zeta$ has maximal dimension.
With no loss of generality, we may suppose that $\{a_1 \zeta, \ldots, a_r \zeta\}$ is linearly independent and
$a_i \zeta=0$ for $r+1 \leq i \leq n$ (compare the proof of \cite[Proposition~3.1]{NaMa13}).
Let $x\in A$ be such that $\varrho(x) \zeta=\zeta $.
Choose a basis  $\{u_1, \ldots, u_r\}$ of $\spa \{a_1, \ldots, a_r\}$ such that
$M\bigl(S_\varrho\varrho(x), \{u_1 \zeta, \ldots, u_r \zeta\}\bigr)$ is triangular.
Set $u_k= a_k$ for $k=r+1, \ldots, n$. Then $\RSrho u_k=0$ for all $k >r$.
It is clear  that $S_\varrho^* S_\varrho^{}\in \CC I$. The proof is complete.
\end{proof}

Our next result is the analogue of Theorem~3.3 in~\cite{NaMa13}.

\begin{thm}\label{spb-Ger}
Let $A$ be a unital Banach algebra, and let ${S}$ be  a spectrally bounded  elementary operator on~$A$.
Let $\varrho$ be an irreducible representation of $A$ on a Banach space~$E$.
Suppose that  $\lDim\LSrho=\Dim\LSrho=n$.
Then $\lDim\VpSrho\leq\frac{n (n-1)}{2}+1$. Moreover, if $\lDim\VpSrho= \frac{n(n-1)}{2}+1$,
then $S_\varrho$ admits  a representation of the form $S_\varrho=\sum_{i=1}^n M_{\varrho(u_i),\varrho(v_i)}$,
where $\varrho (v_i u_j) =0$  for every $i> j$ and $\varrho (v_i u_i) \in \CC I$.
In particular, $\Dim\VpSrho= \frac{n(n-1)}{2}+1$.
\end{thm}
\begin{proof}
Let $S_\varrho= \sum_{i=1}^n M_{\varrho(a_i),\varrho(b_i)}$ and $\lDim\VpSrho= r+1$.
Choose $\zeta\in E$ such that $\Dim\VpSrho\zeta= r+1$ and $\Dim\LSrho\zeta =n$ (Lemma~\ref{free}).
The case $r=0$ is trivial.
Suppose  that $r \neq 0$ and let $\{ \varrho (b_{i_t} a_{j_t}) \zeta \}_{1 \leq t \leq r } \cup \{ \zeta\}$
be a basis of $\VpSrho\zeta$. Pick $x_1, \ldots, x_r \in A$ with
\be
\varrho(x_k b_{i_t} a_{j_t}) \zeta = \delta_{kt} \zeta \quad \text {and} \quad \varrho (x_k) \zeta=0 \quad (1 \leq t, k \leq r).
\ee
Let $N$ be the vector subspace of $ M_n(\mathbb C)$ generated by
$M\bigl(S_\varrho\varrho (x_t), \{\varrho(a_1) \zeta, \ldots, \varrho (a_n) \zeta \}\bigr)$.
It follows from Lemma~\ref{spb-nilp} that $N$ is nilpotent. Since $\Dim N=r$, applying Gerstenhaber's theorem~\cite{Ger},
we get $r \leq \frac{n(n-1)}{2}$ and if $r=  \frac{n(n-1)}{2}$,
there exists a basis $B$ of $\LSrho\zeta$ such that $M(S_\varrho\varrho(x), B)$ is upper triangular for each
$x \in \spa \{ x_1, \ldots, x_r\}$.
Next suppose that $r= \frac{n(n-1)}{2}$  and set $B= \{\varrho(u_1) \zeta, \ldots, \varrho(u_n) \zeta  \}$.
Write  $\varrho{S}= \varrho\bigl(\sum_{i=1}^n M_{u_i,v_i}\bigr)$.
Since $M(S_\varrho\varrho(x), B)= (\pi \varrho(x v_i u_j) \zeta)_{1 \leq i,j \leq n}$, we have $\pi \varrho(x v_i u_j) \zeta=0$ for all
$i \geq j$ and $x \in \spa \{x_1, \ldots, x_r \}$.   Since $\Dim\VpSrho\zeta= \frac{n(n-1)}{2}+1$, the set
$ \{ \varrho(v_i u_j) \zeta: i <j \} \cup \{ \zeta \} $ is a basis of $\VpSrho\zeta$. Moreover,
$\varrho(v_i u_j ) \zeta \in \CC \zeta$ for all $i \geq j$. Arguing similarly to the proof of \cite[Theorem~3.3]{NaMa13},
we show that $\RSrho u_1\subseteq\CC I$ and deduce by induction that $\varrho(v_i u_j) \in \CC I$ for all $i \geq j$.
Thus $\Dim\VpSrho=\frac{n(n-1)}{2}+1$. Hence the set  $\{ \varrho(v_i u_j ): i <j \}\cup \{ I\} $ is a basis of $\VpSrho$,
and  using $\zeta$, we infer   that  $\varrho(v_iu_j ) \in \CC I$ for all $i \geq j$.

Next suppose that $\varrho(v_t u_l)\neq 0$  for some $t >l$ and let $l$ be the smallest integer satisfying this property.
With no loss of generality, we suppose that $\varrho(v_tu_l)= I$. Suppose there is $j > k \geq l$ with the property that
$\varrho(v_j u_l )\neq 0$ and $\varrho(v_k u_l) \neq 0$. Write
\be
\varrho{S}= \varrho (M_{u_1,v_1}+\ldots+M_{u_k,v_k- \lambda v_j}+\ldots +M_{u_j+ \lambda u_k,v_j}+ \ldots + M_{u_n,v_n})
\ee
for some complex number $\lambda$ satisfying $\varrho ((v_k- \lambda v_j)u_l)=0$.
Thus, by choosing the largest~$t$ with $\varrho(v_t u_l) \neq 0$, we may assume that $\varrho(v_k u_l) =0$ for $l \leq k $ and $k \neq t$.
Take $x_k \in\invA$, $x\in A$ such that
\be
\varrho( x_k )\zeta= \frac{1}{k} \varrho(v_l u_t )\zeta,   \; \varrho(x_k v_l u_t) \zeta= \zeta \ \text{and} \
\varrho(x_k v_i u_k) \zeta =\varrho(v_i u_k)\zeta \quad(i <k,\; (i,k) \neq (l,t))
\ee
and
\be
\varrho(x v_l u_t) \zeta=  \varrho(x) \zeta=   \varrho(v_l u_t) \zeta, \; \varrho(x v_i u_k) \zeta=0\quad(i <k,\; (i,k) \neq (l,t)).
\ee
We have
\be
\varrho( x_k^{} x x_k^{-1}) \zeta= \zeta, \; \varrho(x_k^{} x x_k^{-1} v_l u_t) \zeta= k \zeta
\ee
and
\be
\varrho(x_k{} x x_k^{-1} v_i u_k) \zeta=0\quad(i <k,\; (i,k) \neq (l,t)).
\ee
Let $M$ be  the matrix representation of $S_\varrho\varrho(x_k^{} x x_k^{-1})$ with respect to $\{\varrho(u_1)\zeta,\ldots,\varrho(u_n)\zeta\}$.
Then $M$ has the form
\be
 M=   \left(
        \begin{array}{ccc}
          D & 0 & 0\\
          0 & L & 0 \\
          0 & * & * \\
        \end{array}
      \right),
\ee
where $D$ is a diagonal matrix of order $l-1$, $L$ has order $t-l+1$ and is of the form
\be
S=   \left(
      \begin{array}{ccccccc}
        0 & 0&\cdots &0& k  \\
        0& * & \cdots&  0&0   \\
        \vdots& \vdots& \ddots&  0&0\\
        0&*&*&*&0\\
        1& *& *&  *&*\\
        \end{array}
    \right).
  \ee
Computing the characteristic polynomial along the row l, and using the fact that $v_ku_l=0$ for $k >l$ with $k \neq t$, we get
  \be
  P_M= Q+ k  R, \quad  \deg Q= n, \; \deg R=n-2,
\ee
and $Q,\, R$ do not depend on~$k$. Thus, the set
$\{\sigma(S_\varrho\varrho(yxy^{-1})): y\in\invA\}$ is not bounded,  a contradiction.
\end{proof}

\begin{cor}\label{spb-2}
Let $A$ be a unital Banach algebra, and let $S$ be a spectrally bounded elementary operator on $A$ of length~$2$.
Let $\varrho$ be an irreducible representation of $A$ on a Banach space~$E$.
Then either $S$ admits a representation of the form $S= M_{a,b}+ M_{c,d}$, where $\varrho(ba), \varrho(dc) \in \CC I$
and $\varrho(bc)=0$, or $\Dim E=2$ and $S_\varrho$ is similar to the map
\be
z\mapsto \left(
            \begin{array}{cc}
              \lambda  \Tr (z)& * \\
              0 & 0 \\
            \end{array}
          \right)  \quad\text{for some }\lambda\in\CC.
\ee
\end{cor}
\begin{proof}
If $S_\varrho$ has length~$1$, the desired conclusion follows from~\cite{CuMa}; thus we can assume that $\ell(S_\varrho)=2$.
Suppose that $\lDim\LSrho=2$. The case $\lDim\VpSrho=1$ follows from Lemma~\ref{lem:ldimV'=1} and in the other case,
the result follows from Theorem~\ref{spb-Ger}.
Finally suppose  that $\lDim\LSrho= 1$ and set $S_\varrho= \varrho\bigl(M_{a_1,b_1}+M_{a_2,b_2}\bigr)$.
Since $\{\varrho(a_1),\varrho( a_2)\}$ is linearly independent, there exist $\eta \in E$ and $f_1, f_2 \in E^*$
such that $\varrho(a_i)= \eta \otimes f_i$, $i=1,2$; \cite[Theorem~2.3]{BrSe}. Suppose first that $\Dim E \geq 3$.
We claim that $\VSrho=0$. Indeed, suppose for instance that $\rho(b_1)\eta\neq0$.  Then there is  $\zeta \in E$
such that $f_1(\zeta)=1$, $f_2 (\zeta)=0$, and
$\{\zeta, \varrho(b_1) \eta\}$ is linearly independent. Choose $x_k, x \in A$, $k\in\N$ with $x_k$ invertible  satisfying
\begin{equation*}
\varrho(x_k)\zeta=\frac{1}{k}b_1\eta, \; \varrho(x_k b_1) \eta=\zeta, \; \varrho(x) \zeta=b_1 \eta,\;
\varrho(x b_1) \eta=0\text{ \ and \ }\varrho(x_k^{} x x_k^{-1} b_2) \eta \in \CC \zeta.
\end{equation*}
Then $S_\varrho\varrho(x_k^{}xx_k^{-1}) \eta= k\varrho(a_1)\zeta=k \eta$, a contradiction.
Thus $\varrho(b_i a_j)=0$ for all $i,j$.

Next suppose that  $\Dim E=2$ and thus $\rho (A)=M_2 (\mathbb C)$. Since $\lDim\LSrho=1$, we may assume that
$\rho (a_1)= e_{11}$ and $\rho (a_2)= e_{12}$. Hence $S_\varrho=M_{e_{11},b}+M_{e_{12},d}$ for some $b,d\in M_2(\CC)$.
Using Proposition~\ref{tr-spb}, we infer that the matrices $b$ and $d$ have the form
\be
b= \left(
     \begin{array}{cc}
       \lambda & * \\
       0 & * \\
     \end{array}
   \right), \;\;\;
d= \left(
     \begin{array}{cc}
       0 & * \\
       \lambda & * \\
     \end{array}
   \right)  \quad\text{for some }\lambda\in\CC.
\ee
A straightforward computation yields the desired result and completes the proof.
\end{proof}

\begin{cor}\label{Gspb-2}
Let $A$ be a unital Banach algebra, and let $S$ be a spectrally bounded elementary operator on $A$ of length~$2$.
Let $\varrho$ be an irreducible representation of $A$ on a Banach space~$E$.
If $\Dim E\ne2$ then $S^*_\varrho S^{}_\varrho\in\CC I$; if $\Dim E=2$ then either $S^*_\varrho S^{}_\varrho\in\CC I$
or $S^*_\varrho S^{}_\varrho\in\CC(I\otimes\Tr)$.
In either case, $[{S}^* {S} x, x]\in \rada$ for all $x \in A$.
\end{cor}
\begin{proof}
Suppose first that $S_\varrho$ admits a representation of the form $M_{a,b}+M_{c,d}$ such that $\varrho(ba),\varrho(dc)\in\CC I$ and $\varrho(bc)=0$.
Then, for every $z\in\varrho(A)$, $S_\varrho^*S^{}_\varrho z\in\CC z$ and hence $S_\varrho^*S^{}_\varrho\in\CC I$.
On account of Corollary~\ref{spb-2}, if $S_\varrho$ does not have such a representation, then $\Dim E=2$
and $S_\varrho$ can be written in the form $S_\varrho=M_{e_{11},b}+M_{e_{12},d}$, using the same notation
as in the proof of Corollary~\ref{spb-2}. It follows that
\begin{equation*}
S^*_{\varrho} S^{}_{\varrho} z= \left(\begin{array}{cc}
                                                                                \lambda^2 \Tr(z) & 0 \\
                                                                                0 & \lambda^2 \Tr(z) \\
                                                                                \end{array}\right)
                                                                = (\lambda^2\otimes\Tr)(z)\qquad(z\in\varrho(A))
\end{equation*}
and hence $S^*_{\varrho} S^{}_{\varrho}\in\CC I\otimes\Tr$. In either case, $[{S}^* {S} x, x]\in \rada$ for all $x \in A$.
\end{proof}
\begin{cor}\label{spb-2-matrix}
Let ${S}$ be a spectrally bounded  elementary operator of length~$2$ on the matrix algebra $A=M_n (\CC)$,
where $n\geq 3$. Then $(Sx)^3=0$ for every $x \in A$  and ${S}^* {S} =0$.
\end{cor}
\begin{proof}
By Corollary~\ref{spb-2}, $S$ has a representation of the form $S=M_{a,b}+M_{c,d}$ such that $ba, dc \in\CC I$ and $bc=0$.
If $ba\neq 0$ then $b$ is invertible, contradicting $bc=0$. Analogously for $dc$. Thus $ba=dc=bc=0$.
The first assertion hence follows from \cite[Corollary~2.5]{NaMa13} while the second is straightforward.
\end{proof}

\section{Spectrally bounded elementary operators of length~$3$}\label{sect:length3}

\noindent
Like in the purely algebraic setting, \cite{NaMa13},
we have more precise information available for elementary operators of length~$3$. The first result is simply a translation
of Corollary~4.2 in~\cite{NaMa13} into the context of Banach algebras.

\begin{prop}
Let $A$ be a unital Banach algebra, and let ${S}$ be a spectrally infinitesimal elementary operator on $A$ of length~$3$.
Then ${S}^* {S} (A) \subseteq\rada$ and $(Sx)^5 \in \rada$ for every $x \in A$.
\end{prop}

In order to treat the more general situation of spectrally bounded elementary operators,
we first need an auxiliary result.

\begin{lem}\label{lem:suff-cond}
Let $E$ be a Banach space, and let $A$ be a closed irreducible algebra of bounded linear operators on~$E$.
Let ${S}=\sum_{i=1}^3 M_{u_i,v_i}$, $u_i, v_i \in \LE$  be an elementary operator of length~$3$ on~$A$
such that one of the following two cases occurs:

\smallskip
\begin{enumerate}
\item[{\rm (i)}] $(v_iu_j)_{1 \leq i,j \leq 3}= \left(\begin{array}{ccc}
                                             \lambda I & \zeta_1 \otimes f & 0 \\
                                             \zeta_0 \otimes f & \lambda I & \zeta_1 \otimes f \\
                                             0 & - \zeta_0 \otimes f &\lambda I
                                           \end{array}\right); \;  \lambda \in \CC, \; \zeta_0, \zeta_1 \in E \text { and } f \in E^*; $
\smallskip
\item[{\rm (ii)}]                                        $(v_iu_j)_{1 \leq i,j \leq 3}= \left(\begin{array}{ccc}
                                             \lambda I & \zeta_0 \otimes g & 0 \\
                                             \zeta_0 \otimes f & \lambda I & \zeta_0 \otimes g \\
                                             0 & - \zeta_0 \otimes f & \lambda I
                                           \end{array}\right); \; \lambda \in \CC, \;  \zeta_0 \in E \text { and } f,g \in E^*,$
\end{enumerate}
where $E^*$ denotes the dual of~$E$. Then ${S}$ is spectrally bounded.
\end{lem}
\begin{proof}
Case (i): If $\lambda=0$, we are done by  \cite[Theorem 4.1]{NaMa13}.
Suppose next that $\lambda\neq0$. A straightforward computation shows that $S^*1=3\lambda$ and $S^* S= 3 \lambda^2  I$.
Let $x\in A$ and let $\alpha \in\partial\sigma (Sx)$. Then
\begin{equation*}
S^* (Sx- \alpha 1)= 3 \lambda (\lambda x -\alpha 1).
\end{equation*}
Since $Sx- \alpha 1$ is a topological divisor of zero, we infer that $\alpha\in\sigma(\lambda x)$.
It follows that $r(Sx) \leq  |\lambda|r(x)$ for all $x \in A$, as desired.

\smallskip\noindent
Case (ii). This case is treated analogously.
\end{proof}

The following is the main result of this section.

\begin{thm}\label{3-spb}
Let  $E$ be a Banach space  with $\Dim E\geq 4$, and let $A$ be a closed irreducible algebra of bounded linear operators on~$E$.
Let  ${S}\colon A \rightarrow \LE$ be an elementary operator of length~$3$.
Then ${S}$ is spectrally bounded  if and only if there exists a representation of ${S}$ of the form
\begin{equation}
{S}= \sum_{i=1}^3 M_{u_i,v_i},\quad u_i, v_i \in \LE
\end{equation}
such that one of the following three cases occurs:

\smallskip
\begin{enumerate}
\item[{\rm (i)}] $v_i u_j=0$ for all $i > j$ and $v_i u_i \in \CC I$ for all~$i$.
\smallskip
\item[{\rm (ii)}] $(v_iu_j)_{1 \leq i,j \leq 3}= \left(\begin{array}{ccc}
                                             \lambda I & \zeta_1 \otimes f & 0 \\
                                             \zeta_0 \otimes f & \lambda I & \zeta_1 \otimes f \\
                                             0 & - \zeta_0 \otimes f &\lambda I
                                           \end{array}\right); \;  \lambda \in \CC, \; \zeta_0, \zeta_1 \in E \text { and } f \in E^*; $
\smallskip
\item[{\rm (iii)}] $(v_iu_j)_{1 \leq i,j \leq 3}= \left(\begin{array}{ccc}
                                             \lambda I & \zeta_0 \otimes g & 0 \\
                                             \zeta_0 \otimes f & \lambda I & \zeta_0 \otimes g \\
                                             0 & - \zeta_0 \otimes f & \lambda I
                                           \end{array}\right); \; \lambda \in \CC, \;  \zeta_0 \in E \text { and } f,g \in E^*. $
\end{enumerate}
\end{thm}
\begin{proof}
The ``if part'' of the statement follows directly from  Lemma~\ref{lem:suff-cond} and Corollary~\ref{cor:sufficient-cond}.

Now suppose that $S$ is spectrally bounded and $\lDim\LS=3$. If $\lDim\VpS=4$, we are done by Theorem~\ref{spb-Ger}.
Next choose  a separating vector $\zeta \in E $   of  $\LS$ such that $\VpS \zeta$ has maximal dimension. Fix a basis $B$ of $\LS\zeta$.
It follows from Lemma~\ref{spb-nilp} that the set
\be
\mathcal N= \{ M\bigl(Sx, B\bigr): x \VS\zeta \subseteq \CC \zeta, x \zeta=0\}
\ee
is a nilpotent subspace of $M_3(\CC)$. Suppose that $\lDim\VpS=3$. In view of \cite[Proposition 3]{Fas}, we distinguish two cases:

\smallskip\noindent
Case 1:
There exists a representation of ${S}$ of the form  ${S}= \sum_{i=1}^3 M_{u_i,v_i}$
for some $u_i, v_i \in \LE$  such that$\{v_1u_2 \zeta, v_1 u_3 \zeta, \zeta\}$ are linearly independent
and $v_i u_j\zeta\in\CC\zeta$ for all $i\geq j$. Suppose moreover that this fact is true for a dense set of separating vectors of $\LS$ and~$\VpS$.
Then, arguing as in the proof of Theorem~\ref{spb-Ger}, we show that $\RS u_1 \subseteq \CC I$.
We deduce that $v_j u_i\in\CC I$  for all $i \geq j$ and, finally, we show that ${S}= \sum M_{u_i',v_i'}$
where $v'_i u'_j =0$ for $i >j$ and $v'_i u'_i \in \CC I$.

\smallskip\noindent
Case 2:
There exists a representation of ${S}$ of the form  ${S}= \sum_{i=1}^3 M_{u_i,v_i}$
for some $u_i, v_i \in\LE$ (depending on $\zeta$) such that
\be
(v_iu_j \zeta)_{1 \leq i,j \leq 3}= \left(\begin{array}{ccc}
                                             \lambda_{11} \zeta & v_2u_3\zeta & \lambda_{13} \zeta\\
                                             v_2u_1 \zeta& \lambda_{22} \zeta & v_2u_3 \zeta \\
                                             \lambda_{31} \zeta & -v_2u_1 \zeta &\lambda_{33} \zeta
                                           \end{array}\right)\ \  \text{for some} \lambda_{ij} \in \CC.
\ee
Suppose that this assumption is true for every separating vector of $\LS$, with $\VpS \zeta$ of maximal dimension, in an open subset of $X$.
For every $ (i,j) \not\in \{(2,1), (2,3), (1,2), (1,3)   \} $,  the set of operators  $\{v_2 u_1, v_2 u_3, I,  v_i u_j\}$
is locally linearly dependent. Since $(v_i u_j-\lambda_{ij} I) \zeta=0$, we must have
$(v_i u_j-\lambda_{ij} I) E \subseteq \spa \{ v_2 u_1 \zeta, v_2 u_3 \zeta, \zeta\}$.
Set $ v_2 u_1 \zeta= \eta_1$ ,  $ v_2 u_3 \zeta= \eta_2$ and $\zeta= \eta_0$.
Pick  $\eta_3  \in E \setminus \spa \{ \zeta, \eta_1, \eta_2\}$. Write $E= \spa \{\zeta, \eta_1, \eta_2, \eta_3\} \oplus Y$,
for some subspace $Y$ of $E$. Let  $p\colon E \rightarrow Y$ be the natural projection.
Write $v_s u_t = \sum_{k=0}^3 \eta_k \otimes f_{st}^k + pv_su_t$. Then
for every $ (i,j) \not\in \{(2,1), (2,3), (1,2), (1,3)   \}$, we have $v_i u_j = \sum_{k=0}^2 \eta_k \otimes f_{ij}^k$.
Fix non-zero numbers $h_0, \ldots, h_3 \in \CC$ and set $ M(\zeta', h_0, \ldots, h_3)= \bigl(\sum_{k=0}^3 h_k f_{ij}^k (\zeta' )\bigr)_{1 \leq i,j \leq 3}$.
Let $\zeta' \in E$. Then for all but finitely many $\lambda \in \CC$, $\zeta'+ \lambda \zeta$ is a separating vector of~$\LS$.
Fix $\lambda \in \CC$ and choose $x \in A$ satisfying
\be
x \eta_k= h_k (\zeta'+ \lambda \zeta), \  x pv_su_t(\zeta'+ \lambda \zeta)=0\text{ and }  x (\zeta'+ \lambda \zeta)=0\quad(0 \leq k \leq 3).
\ee
Then
\be
M\bigl(Sx, \{u_1 (\zeta'+\lambda\zeta), \ldots, u_3 (\zeta'+ \lambda\zeta)\}\bigr)=M(\zeta'+ \lambda \zeta, h_0, \ldots, h_3)
\ee
which is nilpotent. Arguing as in the proof of \cite[Theorem 4.1]{NaMa13}, we show that
 \be (v_iu_j)_{1 \leq i,j \leq 3}= \left(\begin{array}{ccc}
                                             \lambda_1 I & \zeta_1 \otimes f & \lambda_5 I\\
                                             \zeta_0 \otimes f & \lambda_2 I & \zeta_1 \otimes f \\
                                             \lambda_4 I & - \zeta_0 \otimes f &\lambda_3 I
                                           \end{array}\right),
\ee
where $\lambda_i \in \CC$,  $\zeta_0, \zeta_1 \in E$, $f \in E^*$ and $f(\zeta)=1$.
Choose $x_k \in\invA$ and $x \in A$ such that
\be
x_k \zeta= \frac{1}{k} \zeta_0, \; x_k\zeta_0=\zeta,\; x_k\zeta_1=\zeta_1 \;\text {and }\; x\zeta_0= x\zeta=\zeta_0,\; x\zeta_1=0.
\ee
Then the matrix of ${S} (x_k^{}x x_k^{-1})$ with respect to $\{u_1 \zeta, u_2 \zeta, u_3 \zeta\}$ is:
\[
\left(\begin{array}{ccc}
                                             \lambda_1  &  0  & \lambda_5\\
                                             k  & \lambda_2  & 0  \\
                                             \lambda_4  & - k   &\lambda_3
                                           \end{array}\right).
\]
This entails that $\lambda_5=0$. Proceeding analogously, with $\zeta_1$ instead of $\zeta_0$, we show that $\lambda_4=0$.
Next choose   $x_k\in\invA$ and $x\in A$ such that
\[
x_k \zeta= \frac{1}{k} \zeta_0, \; x_k \zeta_0= \zeta, \; x_k \zeta_1= \zeta_1 \; \text {and }\; x \zeta_0= x \zeta =x \zeta_1= \zeta_0.
 \]
Then the matrix of ${S} (x_k^{}x x_k^{-1})$ with respect to $\{u_1 \zeta, u_2 \zeta, u_3 \zeta\}$ is:
\[
\left(\begin{array}{ccc}
                                             \lambda_1  &  1  & 0 \\
                                             k  & \lambda_2  & 1  \\
                                             0 & - k   &\lambda_3
                                           \end{array}\right).
\]
This implies that  $\lambda_1= \lambda_3$.
Next suppose there are $\zeta, \zeta'\in E$ such that the sets $\{\zeta_0,\zeta_1,\zeta,\zeta'\}$ and
$\{ u_2 \zeta, u_3 \zeta, u_1 \zeta', u_2 \zeta', u_3 \zeta'\}$ are linearly independent.  With no loss of generality,
we may assume that $f(\zeta')=1$ and $f(\zeta)=0$. Pick  $x_k \in\invA$ and $x \in A$ such that
\[
x_k \zeta'= \frac{1}{k}\zeta, \; x_k \zeta_0= \frac{1}{k}\zeta_0, \; x_k \zeta_1= \zeta_1, x_k \zeta= \zeta'
\]
and
\[
x \zeta'= x \zeta= x \zeta_1= \zeta, \;\; x \zeta_0= \zeta'.
\]
Then the matrix representation of ${S} (x_k^{}x x_k^{-1})$ with respect to $\{ u_2 \zeta, u_3 \zeta, u_1 \zeta', u_2 \zeta', u_3 \zeta'\}$ is
\[
\left(
  \begin{array}{ccccc}
    0 & 0 & 1 & 0 & 0 \\
    0 & 0 & 0 & -1 & 0 \\
    0 & 0 & \lambda_1 & 1 & 0 \\
    \lambda_2 k & 0 & 0 & \lambda_2 & 1 \\
    0& \lambda_1 k & 0 & 0 & \lambda_1\\
  \end{array}
\right)
\]
Its characteristic polynomial is
\be
t^2 (\lambda_1-t)^2 (\lambda_2-t)+(\lambda_2- \lambda_1)k(\lambda_1-t) t
\ee
hence, we must have $\lambda_2= \lambda_1$. We get the same conclusion if we swop $u_1$ and~$u_3$.

Now suppose that for every $\zeta, \zeta' \in E$, the sets $\{ u_2 \zeta', u_k \zeta', u_1 \zeta, u_2 \zeta, u_3 \zeta\}$, $k=1,3$
are linearly dependent. Then, for every separating vector $\zeta$ of $\LS$ and $k \in \{1,3\}$, either there are
$\alpha_k, \beta_k$ such that $(\alpha_k u_k + \beta_k u_2 )E \subseteq \LS \zeta$ or
$u_k, u_2$ are rank one modulo $\LS\zeta$.  Suppose towards a contradiction that $\lambda_1\neq\lambda_2$.
Suppose that $E$ is finite dimensional. Since $v_1u_3 =0$, $v_1$ cannot be injective. Thus $\lambda_1=0$.
Clearly, we can assume that $\lambda_2=1$. Since $v_2, u_2$ must be bijective, $u_1, u_3, v_1, v_3$ are rank~$1$.
Set $u_k= \eta_k \otimes g_k$, $k=1,3$. Then, with no loss of generality, we may assume that $g_k=f$.
Since $u_2$ is bijective and $v_2u_2= I$, $u_2\zeta_0=\eta_1$ and $u_2\zeta_1=\eta_3$.  On the other hand,
\be
v_1u_2 \zeta_0= f(\zeta_0) \zeta_1= v_1 \eta_1=0.
\ee
Thus, $f(\zeta_0)=0$.  Analogously, we get $f(\zeta_1)=0$. Pick $\eta\in E$ such that $f(\eta)=1$ and $\{\eta_1, \eta_2, u_2 \eta\}$
is linearly independent. Choose $x_k \in\invA$, and $x \in A$ with
\be
x_k \zeta_1= \frac{1}{k} \zeta_0, x_k \zeta_0= \zeta_1, x_k \eta=\eta, x \zeta_0= \eta= x \eta, x \zeta_1= \zeta_0.
\ee
Then
\be
M\bigl( {S} (x_k x x_k^{-1}), \{\eta_1, \eta_3, u_2 \eta\}\bigr)= \left(
                                                           \begin{array}{ccc}
                                                             0 & 0 & 1 \\
                                                             k & 0 & 0 \\
                                                             0 & 1 & 1 \\
                                                           \end{array}
                                                         \right).
\ee
This yields a contradiction, as desired. Now suppose that $E$ has infinite dimension. Suppose first that $\lambda_2\neq 0$.
Then $u_2$ is injective, hence $u_1, u_3$ have rank one modulo $\LS \zeta$, for each separating vector $\zeta$ of $\LS$.
This implies (using again the injectivity of $u_2$) that $\Dim (u_1 E+u_3 E ) \leq 2$. Applying $v_2$ to $u_1$ and $u_3$,
we infer that $u_1, u_3$ are rank~$1$. Thus $\lambda_1=0$. Pick $\eta, \eta'$ such that the sets
$\{u_1 \eta, u_2 \eta, u_3 \eta, u_2 \eta'\}$ and $ \{\eta, \eta', \zeta_0, \zeta_1\}$ are linearly independent.
With no loss of generality, we can suppose that $u_1= u_1 \eta \otimes f$, $u_3= u_3 \eta \otimes f$, $f(\eta)=1$
and $f(\eta')=0$. Choose $x_k \in \invA,\, x \in A$ with
\be
x_k \eta= \frac{1}{k}\zeta_0, x_k \zeta_0= \eta, x_k \zeta_1=\eta+\eta', x_k \eta'= \zeta_1
\ee
and
\be
x \eta=\zeta_0, x \eta'= \zeta_1, x \zeta_0= \zeta_0, x \zeta_1= 2 \zeta_0.
\ee
Then
\be
M\bigl({S} (x_k x x_k^{-1}), \{u_1 \eta, u_3 \eta, u_2 \eta, u_2 \eta'\}\bigr)= \left(
                                                                      \begin{array}{cccc}
                                                                        0 & 0 & 1 & 0 \\
                                                                        0 & 0 & -k & 0 \\
                                                                        k & 1 & 1 & 1 \\
                                                                        0 & 1 & 0 & 0 \\
                                                                      \end{array}
                                                                    \right)
\ee
Computing the characteristic polynomial, we get a contradiction. Thus, this case cannot occur.

Next suppose that $\lambda_1 \neq 0$, say $\lambda_1=1$. Then $u_1, u_3$ are injective.
As above, we see that $u_2$ must have rank one. Set $u_2= z \otimes g$, then we can assume that $g=f$, $v_1 z = \zeta_1$
and $v_3 z= - \zeta_0$. Choose $\eta, \eta'$ such that $\{u_1 \eta, u_3 \eta, u_2 \eta, u_1 \eta'\}$ is linearly independent.
Assume that $u_2 \eta=z$ and $f(\eta')=0$.
Choose $x, x_k$ as above to obtain
\be
M\bigl({S} (x_k x x_k^{-1}), \{u_1 \eta, u_3 \eta, u_2 \eta, u_1 \eta'\}\bigr)= \left(
                                                                      \begin{array}{cccc}
                                                                        1 & 0 & 1 & 1 \\
                                                                        0 & 1 & -k & 0 \\
                                                                        k & 1 & 0 & 0 \\
                                                                        0 & 0 & 1 & 0 \\
                                                                      \end{array}
                                                                    \right).
\ee
Once again, we get a contradiction using the characteristic polynomial.

Now suppose that $\lDim\VpS=2$. Similar arguments as above yield that
${S}= \sum_{i=1}^3 M_{u_i,v_i}$, $u_i, v_i \in \LE$ with either
$v_i u_j=0$ for all $i > j$ and $v_i u_i \in \CC I$ for all~$i$, or
\be
(v_iu_j)_{1 \leq i,j \leq 3}= \left(\begin{array}{ccc}
                                             \lambda_{1} I & \zeta_0 \otimes g  & \lambda_4 I\\
                                             \zeta_0 \otimes f & \lambda_2 I & \zeta_0 \otimes g \\
                                             \lambda_5 I & - \zeta_0 \otimes f &\lambda_3 I
                                           \end{array}\right);
\ee
where $  \lambda_i  \in \CC$, $\zeta_0 \in E$ and  $f, g \in E^*$. This case is treated analogously.

Suppose now that $\lDim \LS=2$.  Suppose first that there exists $b \in \RS$ with $b \LS=0$,
then we can write ${S}= M_{a,b} + \sum_{i=1}^2 M_{u_i,v_i}$ for suitable $u_i, v_i$.
Our arguments show that $\sum_{i=1}^2 M_{u_i,v_i}$ must be spectrally bounded. Thus we get easily the desired conclusion.

Next suppose the contrary. Write $ {S}= \sum_{i=1}^3 M_{a_i,b_i}$.  Using ~\cite{ChSe}, we distinguish 3 cases.

\smallskip\noindent
Case 1: $\Dim \LS E =2$.
Let $\zeta \in E$ be such that $\LS \zeta$ has maximal dimension. Certainly we may assume that  $a_3 \zeta=0$. Using the set
\be
\{ Sx_{|\LS \zeta}: x \VS \zeta \subseteq \CC \zeta, x \zeta=0 \}
\ee
we see that there exists $b \in \RS$  such that $ba_i E \subseteq \CC \zeta$.
Consider $\zeta_1, \zeta_2, \zeta_3$ linearly independent separating vectors of~$\LS$.
Then there exists $b'_1, b'_2, b'_3$ with $b'_i \LS E \subseteq \CC \zeta_i$.  Therefore $\{b'_1, b'_2, b'_3\}$ is linearly independent.
Thus, the dimension of $E$ cannot exceed~$3$. This case cannot occur.

\smallskip\noindent
Case 2:
The vector space $\LS$ is standard in the sense of~\cite{ChSe}.
Fix $u \in \LS E$.  We claim that  either $\Dim \RS u\leq 1$ or $\Dim \RS u=3$.
Choose $\zeta \in E$ such that $u= a \zeta$ for some $a\in\LS$.  If $\zeta \not\in\RS u$, choose $x \in A$
such that $x \zeta=0$ and $x \VS\zeta \subseteq \CC \zeta$. Then $Sx$ must be nilpotent by Lemma~\ref{spb-nilp}.
A straightforward argument shows  that $\Dim\RS u\leq1$.
Next suppose that $\Dim \RS u \geq 2$. Then $\zeta \in \RS u$ and $\ker a + \CC \zeta \subseteq \RS u$.
Since $\Dim aE=2$ and $\Dim \RS u \leq 3$, we must have $\Dim E=4$ and $\Dim\RS u=3$, as desired.  Now write
\be
a_1= u_1 \otimes f_2 + u_2 \otimes f_3, \; a_2= - u_1 \otimes f_1+ u_3 \otimes f_3, a_3= u_2 \otimes f_1 + u_3 \otimes f_2,
\ee
where  $f_1, f_2, f_3$ are linearly independent functionals and  $ u_1,u_2, u_3$ are linearly independent vectors of~$E$.
Then, for every $x \in A$, we have
\be
Sx= u_1 \otimes (f_2 xb_1-f_1 x b_2) + u_2 \otimes (f_3 x b_1+ f_1 x b_3)+ u_3 \otimes (f_3 x b_2+ f_2 x b_3).
\ee
Choose $e_j \in E$ with $f_i (e_j)= \delta_{ij}$.
Fix $j \in \{1,2,3\}$. Let $k,l$ be such that $\{j,k,l\}= \{1,2,3\}$. The vector subspace
\be
\{M\bigl(Sx, \{a_ke_j, a_l e_j\}\bigr): x \in A, xe_j=0, x \VS e_j \subseteq \CC e_j\}
\ee
is nilpotent. Thus there exists $z_j \in \spa\{a_ke_j,a_l e_j\}$ such that $ \spa \{b_k, b_l\} z_j \subseteq \CC e_j$.
Since $\Dim \RS z \neq 2$ for every $z \in \LS E$, we must  have $\Dim \RS z \leq 1$ for every $z \in \LS E$.
Now using Proposition~\ref{tr-spb}, we get the desired conclusion.

\smallskip\noindent
Case 3: The vector space $\LS$ is not minimal linearly dependent.

With no loss of generality, we may suppose that there exists $\eta \in E$ such that $a_k= \eta \otimes f_k$ for $k=1,2$.  Then for all $x \in A$ we have
\be
Sx= \eta \otimes \sum_{k=1}^2f_kxb_k+a_3xb_3.
\ee
Suppose that $b_3a_3 \not\in \CC I$. Choose $\zeta \in E$ such that the two sets  $\{b_3 a_3 \zeta, \zeta\}$
and $\{a_1 \zeta, a_3 \zeta\}$  are linearly independent. The vector space
\be
\mathcal N= \{M\bigl(Sx, \{a_1 \zeta, a_3 \zeta\}\bigr): x \in A,\; x \zeta=0, \;x b_3 a_3 \zeta, x b_3 \eta \in \CC \zeta \}
\ee
is nilpotent. Hence $\Dim \mathcal N \leq 1$.  Since for every  $x \in A$ with  $x b_3 a_3 \zeta, x b_3 \eta \in \CC \zeta$, we have
\be
M\bigl(Sx, \{a_1 \zeta, a_3 \zeta\}\bigr)= \left(
                                    \begin{array}{cc}
                                      \sum_{k=1}^2f_k(xb_k \eta) & \sum_{k=1}^2f_k(xb_k a_3 \zeta) \\
                                      \pi x b_3 \eta & \pi x b_3 a_3 \zeta \\
                                    \end{array}
                                  \right),
\ee
(here we suppose that $a_1 \zeta= \eta$).
Thus $b_k \eta \in \spa \{b_3a_3 \zeta, \zeta\}$, for $ 1 \leq k \leq 3$.
Now assume that $b_3 \eta \not\in \CC \zeta$. Then, replacing $a_3$ with $a_3+ \lambda a_1+ \lambda' a_2$ if needed,
we see that we  can assume that $b_1 \eta, b_2 \eta \in \CC \zeta$. It follows from Proposition~\ref{tr-spb}
that $b_3a_3 \zeta \in \CC \zeta$, a contradiction. Next suppose that $b_3 \eta \in \CC \zeta$.
With no loss of generality, we can suppose that $b_2 \eta \in \CC \zeta$. It is easy to infer that $b_1 \eta=0$,
and once again we get $b_3 a_3 \zeta \in \CC \zeta$, which is not possible. Hence,  we have $b_3 a_3\in \CC I$. Since we can replace
$a_3$ by $a_3+ \lambda a_1+ \lambda' a_2$, for every $\lambda, \lambda' \in \CC$, we get $b_3 a_k \in \CC I$ for each~$k$.
But $a_1, a_2$ are rank one, therefore $b_3a_1=b_3a_2=0$.  Using again the set $\mathcal N$ we see that
$b_1 \eta, b_2 \eta \in \CC \zeta$. Hence, $b_1 \eta= b_2 \eta=0$. The proof of this case is complete.

\smallskip
Finally, assume that $\lDim \LS=1$.   Then there exists $\eta \in E, f_i \in E^*$ such that $a_i=\eta \otimes f_i$.
Since $\Dim E \geq 4$, we get easily the desired conclusion using Proposition~\ref{tr-spb}.
\end{proof}
The exceptional cases $\Dim E=2$ and $\Dim E=3$ do not seem to allow a concise description (such as in Corollary~\ref{spb-2})
and thus will not be discussed here.

\begin{acknowledgement}
Part of the research on this paper was carried out while the second-named author visited the Universit\' e Moulay Ismail
in Morocco. He would like to express his gratitude for the generous hospitality and support received from his colleagues there.
This visit was supported by the London Mathematical Society under their ``Research in Pairs'' programme.
\end{acknowledgement}

\end{document}